\documentclass[leqno]{article}

\usepackage{layout}
\usepackage[OT1]{fontenc}
\usepackage{amsthm,amsmath,graphicx,latexsym,amssymb,amscd,amsfonts,enumerate, bm}
\usepackage[colorlinks,citecolor=blue,linkcolor=red,urlcolor=blue]{hyperref}

\theoremstyle{plain}
\newtheorem{theorem}{Theorem}
\newtheorem{proposition}[theorem]{Proposition}

\newtheorem{remark}[theorem]{Remark}

\newcommand{\bbE}{\mathbb{E}}
\newcommand{\bbR}{\mathbb{R}}
\newcommand{\mL}{\mathcal{L}}

\makeatletter
\def\section{\@startsection {section}{1}{\z@}{-3.5ex plus -1ex minus 
-.2ex}{2.3ex plus .2ex}{\normalsize\bf}}
\makeatother

\makeatletter
\def\subsection{\@startsection {subsection}{1}{\z@}{-3.5ex plus -1ex minus 
-.2ex}{2.3ex plus .2ex}{\normalsize\bf}}
\makeatother

\begin{document}

\title{\textbf{\large On time-changed Gaussian processes and their associated Fokker-Planck-Kolmogorov equations}}
\author{\textsc{\normalsize Marjorie G.\ Hahn, Kei Kobayashi, Jelena Ryvkina and Sabir Umarov}\thanks{Department of Mathematics, 
                        Tufts University, 503 Boston Avenue, Medford, MA 02155, USA; marjorie.hahn@tufts.edu, kei.kobayashi@tufts.edu, jelena.ryvkina@tufts.edu, sabir.umarov@tufts.edu}
        \vspace{1mm} \\ \textit{\small Tufts University}}
\date{ }
\maketitle                                                       
\vspace{-6mm}
\renewcommand{\thefootnote}{\fnsymbol{footnote}}

\begin{abstract}
This paper establishes Fokker-Planck-Kolmogorov type equations for time-changed
Gaussian processes.    
Examples include those equations for a time-changed fractional Brownian motion with time-dependent Hurst parameter and for a time-changed Ornstein-Uhlenbeck process.  The time-change process considered is the inverse of either a stable subordinator or a mixture of independent stable subordinators. 
\end{abstract}

\section{Introduction} 
\label{Sec1}

A one-dimensional stochastic process $X=(X_t)_{t\ge 0}$ is called a \textit{Gaussian process} if the random vector $(X_{t_1},\ldots,X_{t_m})$ has a multivariate Gaussian distribution for all finite sequences $0\le t_1<\cdots<t_m<\infty$.
The joint distributions are characterized by the mean function $\bbE[X_t]$ and the covariance function $R_X(s,t)=\textrm{Cov}(X_s,X_t)$.  
The class of Gaussian processes contains some of the most important stochastic processes in both theoretical and applied probability, including Brownian motion, fractional Brownian motion \cite{book,Cheridito03,turbulence,solar}, and Volterra processes \cite{AlosNualart01,Decr-volterra}.

In this paper, Fokker-Planck-Kolmogorov type 
equations (FPKEs) for time-changed Gaussian processes are derived.   These FPKEs are partial differential equations (PDEs) satisfied by the transition probabilities of those processes. 
Main features of these equations involve 1){\nolinebreak} fractional order derivatives in time and 2) a new class of operators acting on the time variable; see \eqref{FP-volterra3-2}--\eqref{FP-volterra3.5}.    
By definition, the \textit{fractional order derivative $D_{\ast}^{\beta}$ of order $\beta\in (0,1)$ in the sense of Caputo-Djrbashian} is given by
\begin{equation}
\label{Caputo}
  {D}_{\ast}^{\beta} g(t)={D}_{\ast,t}^{\beta} g(t) =
      \frac{1}{\Gamma(1-\beta)}
      \int_0^{t} \frac{g^{\prime}(\tau) }{(t-\tau)^{\beta}}\hspace{1pt} d \tau, 
\end{equation}
with $\Gamma(\cdot)$ being Euler's Gamma function. 
By convention, set $D^1_{\ast,t}=d/dt$.  
Introducing the fractional integration operator
\[
J^{\alpha}g(t)=J^{\alpha}_t g(t)=\frac{1}{\Gamma(\alpha)}\int_0^{t}
(t-\tau)^{\alpha-1}g(\tau)\hspace{1pt}d\tau, ~~ \alpha >0,
\]
one can represent $D_{\ast,t}^{\beta}$ in the form
$D_{\ast,t}^{\beta}=J^{1-\beta}_t\circ (d/dt)$ (see \cite{GM97} for details).

For the last few decades, time-fractional order FPKEs have appeared as an essential tool for the study of dynamics of various complex processes arising in anomalous diffusion in physics
\cite{MetzlerKlafter00,Zaslavsky}, finance \cite{GMSR,Janczura}, hydrology
\cite{BWM} and cell biology \cite{Saxton}.
Using several different methods, many authors derive FPKEs associated with time-changed stochastic processes.
For example, in \cite{MMM-TAMS}, FPKEs for time-changed continuous Markov processes are obtained via the theory of semigroups.   
Paper \cite{HKU} identifies a wide class of stochastic differential equations whose associated FPKEs are represented by time-fractional distributed order pseudo-differential equations.  
The driving processes for these stochastic differential equations are time-changed L\'evy processes.   
Two different approaches are taken, one based on the semigroup technique and the other on the time-changed It\^o formula in \cite{Kobayashi}.   
Paper \cite{HKU-2} provides FPKEs associated with a time-changed fractional Brownian motion from a functional-analytic viewpoint.  
Continuous time random walk-based approaches to derivations of time-fractional order FPKEs are illustrated in \cite{GM1,MMM-ctrw,UG,US}.   
In the current paper, we follow the method presented in \cite{HKU-2}.

Consider a time-change process $E^\beta=(E^\beta_t)_{t\ge 0}$ given by the \textit{inverse}, or the \textit{first hitting time process}, of a $\beta$-stable subordinator $W^\beta=(W^\beta_t)_{t\ge 0}$ with $\beta\in(0,1)$.  The relationship between the two processes is expressed as 
$E^\beta_t=\inf\{s>0\hspace{1pt};\hspace{1pt} W^\beta_s>t \}$.
To make precise the problem being pursued in this paper, first recall the following results.  For proofs of these results as well as analysis on the associated classes of stochastic differential equations, 
see e.g.\ \cite{GM1,HKU,HKU-2,MMM-ctrw}.
Throughout the paper, all processes are assumed to start at 0.
\begin{enumerate}
   \item[(a)] If $E^\beta$ is independent of an $n$-dimensional Brownian motion $B$, then $B$ and the time-changed Brownian motion $(B_{E^\beta_t})$ have respective transition probabilities $p(t,x)$ and $q(t,x)$ satisfying the PDEs
\begin{align}\label{Brownian FPK}
   \partial_t \hspace{1pt}p(t,x)=\frac 12 \varDelta \hspace{1pt}p(t,x) \ \ \ \textrm{and} \ \ \  D_{\ast,t}^{\beta}\hspace{1pt}q(t,x)=\frac 12 \varDelta \hspace{1pt}q(t,x),  
\end{align}
where $\partial_t=\frac{\partial}{\partial t}$ and $\varDelta=\sum_{j=1}^n \partial_{x^j}^2=\sum_{j=1}^n \bigl(\frac{\partial}{\partial {x^j}}\bigr)^2$, with the vector $x\in\bbR^n$ denoted as $x=(x^1,\ldots,x^n)$.
   \item[(b)]  Let $L$ be a L\'evy process whose characteristic function is given by $\bbE[e^{i(\xi,L_t)}]=e^{t \psi(\xi)}$ with symbol $\psi$ (see \cite{Applebaum,Sato}).     
If $E^\beta$ is independent of $L$, then $L$ and the time-changed L\'evy process $(L_{E^\beta_t})$ have respective transition probabilities $p(t,x)$ and $q(t,x)$ satisfying the PDEs 
\begin{align}\label{Levy FPK}
   \partial_t \hspace{1pt}p(t,x)=\mathcal{L}^\ast  p(t,x) \ \ \ \textrm{and} \ \ \ 
   D_{\ast,t}^{\beta}\hspace{1pt}q(t,x)=\mathcal{L}^\ast q(t,x), 
\end{align}
where $\mathcal{L}^\ast$ is the conjugate of the pseudo-differential operator with symbol $\psi$.   
   \item[(c)] If $E^\beta$ is independent of an $n$-dimensional fractional Brownian motion $B^H$ with Hurst parameter $H\in(0,1)$ (definition is given in Example \hyperlink{Application-1}{1}), then $B^H$ and the time-changed fractional Brownian motion $(B^H_{E^\beta_t})$ have respective transition probabilities $p(t,x)$ and $q(t,x)$ satisfying the PDEs
\begin{align}\label{FBM FPK}
   \partial_t \hspace{1pt}p(t,x)=H t^{2H-1}\varDelta \hspace{1pt}p(t,x) \ \ \ \textrm{and} \ \ \ 
   D_{\ast,t}^{\beta}\hspace{1pt}q(t,x)=H G^\beta_{2H-1,t}\varDelta \hspace{1pt}q(t,x), 
\end{align}
where $G^\beta_{\gamma,t}$ with $\gamma\in(-1,1)$ is the operator acting on $t$ given by
\begin{equation} \label{presentation}
G^\beta_{\gamma,t} \hspace{1pt}g(t)= \beta \Gamma(\gamma + 1) J^{1-\beta}_t
\mathcal{L}^{-1}_{s \to t}\biggl[\frac{1}{2\pi i} \int_{C-i
\infty}^{C+i\infty}
\frac{\tilde{g}(z)}{(s^{\beta}-z^{\beta})^{\gamma+1}}\hspace{1pt}dz \biggr](t),
\end{equation}
with $0 < C <s$ and $z^{\beta}=e^{\beta \hspace{1pt} \textrm{Ln} (z)},$
$\textrm{Ln}(z)$ being the principal value of the complex logarithmic function $\ln(z)$ with cut along the negative real axis.
Here $\tilde{g}(s)=\mathcal{L}_{t\to s}[g(t)](s)=\mathcal{L}_{t}[g(t)](s)$ and $\mathcal{L}_{s\to t}^{-1}[f(s)](t)$ denote the Laplace transform and the inverse Laplace transform, respectively. 
\end{enumerate}  

 Obviously (a) provides a special case of both (b) and (c).  Note that in \eqref{Brownian FPK} and \eqref{Levy FPK} the second PDE  
is obtained upon replacing the first order time derivative $\partial_t$
in the first PDE
by the fractional order derivative $D^\beta_{\ast,t}$ and the right-hand side remains unchanged.   On the other hand, as \eqref{FBM FPK} shows, this is not the case for a time-changed fractional Brownian motion; the right hand side of the time-fractional order FPKE has a different form than that of the FPKE for the untime-changed fractional Brownian motion.
Namely, the operator $G^\beta_{2H-1,t}$ instead of $t^{2H-1}$ appears, which is ascribed to dependence between increments over non-overlapping intervals of the fractional Brownian motion $B^H$.   
When $H=1/2$ so that the fractional Brownian motion coincides with a usual Brownian motion, the FPKEs in \eqref{FBM FPK} match with those in \eqref{Brownian FPK}.   
The operators $\{G^\beta_{\gamma,t}; \gamma\in(-1,1)\}$ are known to satisfy the semigroup property (Proposition 3.6 in \cite{HKU-2}).

Following the functional-analytic technique presented in \cite{HKU-2} to obtain the time-fractional order FPKE in \eqref{FBM FPK}, we will establish in Theorem \ref{theorem_FP-volterra} the FPKE for the time-changed Gaussian process $(X_{E^\beta_t})$ under the assumption that the time-change process $E^\beta$ is independent of the Gaussian process $X$.
Moreover, generalization to time-changes which are the inverses of mixtures of independent stable subordinators will be considered in Theorem \ref{theorem_FP-volterra_general}. 
In Section \ref{Sec4}, applications of these results yield FPKEs for time-changed mixed fractional Brownian motions as well as those for time-changed Volterra processes.  
Fractional Brownian motions with time-dependent Hurst parameter $H=H(t)\in(1/2,1)$ are among the Volterra processes considered; see Example \hyperlink{Application-3}{3}.
Two equivalent forms of FPKEs for a time-changed Ornstein-Uhlenbeck process are compared in Example \hyperlink{Application-5}{5}.

\section{Preliminaries}
\label{Sec2}

Let $E^\beta$ be the inverse of a stable subordinator $W^\beta$ starting at $0$ with stability index $\beta \in(0,1)$.  The process $W^\beta$ is a L\'evy process with Laplace transform $\mathbb{E}[e^{-sW^\beta_t}]=e^{-ts^{\beta}}$ and self-similarity: $(W^\beta_{ct})_{t\ge 0}=(c^{1/ \beta} W^\beta_t)_{t\ge 0}$ in the sense of finite dimensional distributions for all $c>0$ (see \cite{Applebaum,Sato}).   
Since $W^\beta$ is strictly increasing, its inverse $E^\beta$ is
continuous and nondecreasing, but no longer a L\'evy process (see \cite{MMM-infinitemean}).  
The density $f_{E_t^\beta}$ of $E_t^\beta$ can be expressed using the density $f_{W_\tau^\beta}$ of $W_\tau^\beta$ as 
\begin{align*} 
f_{E_t^\beta}(\tau) = \partial_\tau \mathbb{P}(E^\beta_t\le \tau)
=\partial_\tau \bigl\{1-\mathbb{P}(W^\beta_\tau < t)\bigr\}
=-\partial_\tau \bigl\{[J^1_t f_{W_\tau^\beta}](t) \bigr\},
\end{align*}
from which it follows that
\begin{align}\label{laplace_Ebeta}
   \mathcal{L}_{t \to s}[f_{E_t^\beta}(\tau)](s)= -\partial_\tau \biggl[\frac{\widetilde{f_{W_\tau^\beta}}(s)}{s}\biggr]
=s^{\beta-1} e^{-\tau s^\beta}, ~ s>0, ~ \tau \ge 0.
\end{align}
The function $f_{E_t^\beta}(\tau)$ is $C^{\infty}$ with respect to the two variables $t$ and $\tau$.

The notion of time-change can be extended to the more general case where
the time-change process is given by the inverse of an arbitrary
mixture of independent stable subordinators. 
Let $\rho_\mu(s)=\int_0^1
s^{\beta}d \mu (\beta),$ where 
$\mu$ is a finite measure with $\textrm{supp}\, \mu \subset (0,1).$
Let $W^{\mu}$ be a nonnegative stochastic
process satisfying $\mathbb{E}[ e^{-sW^{\mu}_t}]=e^{-t \rho_\mu(s)}$ and let
$E^{\mu}_t=\inf \{\tau>0\hspace{1pt};\hspace{1pt} W^{\mu}_{\tau} > t\}.$
Clearly, $W^\mu=W^{\beta_0}$ if $\mu=\delta_{\beta_0}$, the Dirac measure on $(0,1)$ concentrated on a single point $\beta_0$.  
The process
$W^{\mu}$ represents a weighted mixture of independent stable subordinators.    
Similar to the identity \eqref{laplace_Ebeta}, the density $f_{E_t^\mu}$ of $E_t^\mu$ has Laplace transform
\begin{align}\label{laplace_Emu}
   \mathcal{L}_{t \to s}[f_{E_t^\mu}(\tau)](s)
=\frac{\rho_\mu(s)}{s} e^{-\tau \rho_\mu(s)}, ~ s>0, ~ \tau \ge 0.
\end{align}
For further properties of $f_{E_t^\beta}(\tau)$ and $f_{E_t^\mu}(\tau)$, see \cite{HKU-2,MMM-ctrw}.

The above time-change process $E^\mu$ is connected with the \textit{fractional derivative $D^\mu_{\ast}$ with distributed orders} given by 
\begin{align}\label{distributed-derivative}
   D^\mu_{\ast}\hspace{1pt}g(t)=D^\mu_{\ast,t}\hspace{1pt}g(t)=\int_0^1 D^\beta_\ast g(t)\hspace{1pt} d\mu(\beta).   
\end{align}
Namely, if $E^\beta$ is replaced by $E^\mu$ in items (a) and (b) of the list of known results in Section \ref{Sec1}, then the FPKEs for the time-changed Brownian motion $(B_{E^\mu_t})$ and the time-changed L\'evy process $(L_{E^\mu_t})$ are respectively given by (see \cite{HKU})
\[
   D^\mu_{\ast,t} \hspace{1pt}q(t,x)=\frac 12 \varDelta\hspace{1pt} q(t,x)
   \ \ \ \textrm{and} \ \ \  
   D^\mu_{\ast,t} \hspace{1pt}q(t,x)=\mL^\ast q(t,x).  
\] 
For investigations into PDEs with fractional derivatives with distributed orders, see \cite{Kochubey,MMM-ultraslow,UG05}.

The covariance function of a given zero-mean Gaussian process is symmetric and positive semi-definite; conversely, every symmetric, positive semi-definite function on $[0,\infty)\times[0,\infty)$ is the covariance function of some zero-mean Gaussian process (see e.g., Theorem 8.2 of \cite{Janson}).  Examples of such functions include $R_X(s,t)=s \wedge t$ for Brownian motion and $R_X(s,t)=\sigma_0^2+s\cdot t$ which is obtained from linear regression
 (see \cite{RW}).  
The sum and the product of two covariance functions for Gaussian processes are again covariance functions for some Gaussian processes.  For more examples of covariance functions, consult e.g.\ \cite{RW}.

An $n$-dimensional Gaussian process $X=(X^1,\ldots,X^n)$ is a process whose components $X^j$ are independent one-dimensional Gaussian processes with possibly distinct covariance functions $R_{X^j}(s,t)$.
The variance functions $R_{X^j}(t)=R_{X^j}(t,t)$ will play
an important role in establishing FPKEs for Gaussian and time-changed Gaussian processes.  
For differentiable variance functions $R_{X^j}(t)$, let
\begin{align}\label{FP-volterra0.5}
   A=A_X=\frac 12\sum_{j=1}^n R_{X^j}'(t)\hspace{1pt}\partial_{x^j}^2.
\end{align}
Clearly $A=\frac{1}{2}\hspace{1pt}\varDelta$ if $X$ is an $n$-dimensional Brownian motion.

\begin{proposition} \label{proposition_FP-volterra}
Let $X=(X^1,\ldots,X^n)$ be an $n$-dimensional zero-mean Gaussian process with covariance functions $R_{X^j}(s,t)$, $j=1,\ldots,n$.  Suppose the variance functions $R_{X^j}(t)=R_{X^j}(t,t)$ are differentiable on $(0,\infty)$.  
Then the transition probabilities $p(t,x)$ of $X$ satisfy the PDE
\begin{align} \label{FP-volterra1}
   \partial_t \hspace{1pt}p(t,x)= A\hspace{1pt} p(t,x), \ t>0, \ x\in\bbR^n,
\end{align}
where $A$ is the operator in \eqref{FP-volterra0.5}.
\end{proposition}

\begin{proof}
Since the components $X^j$ of $X$ are independent zero-mean Gaussian processes, it follows that
\begin{align}
p(t,x)=\prod_{j=1}^n\bigl(2\pi R_{X^j}(t)\bigr)^{-1/2}\times \exp\biggl\{-\sum_{j=1}^n\frac{(x^j)^2}{2R_{X^j}(t)}\biggr\}.
\end{align}
Direct computation of partial derivatives of $p(t,x)$
yields the equality in \eqref{FP-volterra1}.
\end{proof}

\begin{remark}
\begin{em}
a)  In Proposition \ref{proposition_FP-volterra}, if the components $X^j$ are independent Gaussian processes with a \textit{common} variance function $R_{X}(t)$ which is differentiable, then 
\eqref{FP-volterra1} reduces to the following form which agrees with the classical FPKE:
\begin{align} 
   \partial_t \hspace{1pt}p(t,x)=\frac 12 R_X'(t) \varDelta \hspace{1pt}p(t,x).
\end{align}

b) If $X=(X^1,\ldots,X^n)$ is an $n$-dimensional Gaussian process with mean functions $m_{X^j}(t)$ and covariance functions $R_{X^j}(s,t)$, and if both $m_{X^j}(t)$ and $R_{X^j}(t)=R_{X^j}(t,t)$ are differentiable,  
then the associated FPKE contains an additional term:
\begin{align} 
   \partial_t \hspace{1pt}p(t,x)=A \hspace{1pt}p(t,x)
   + B \hspace{1pt}p(t,x) \ \ \ \ \textrm{where} \ \ \ \ 
    B=-\sum_{j=1}^n m_{X^j}'(t)\hspace{1pt} \partial_{x^j}.
\end{align}
Such Gaussian processes include e.g.\ the process defined by the sum of a Brownian motion and a deterministic differentiable function.

c) The initial distribution of the Gaussian process $X$ in Proposition \ref{proposition_FP-volterra} is not given; however, it needs to be specified in order to guarantee uniqueness of the solution to PDE \eqref{FP-volterra1}.  The same argument applies to the PDEs to be established in Theorems \ref{theorem_FP-volterra} and \ref{theorem_FP-volterra_general} as well.
\end{em}
\end{remark}

\section{FPKEs for time-changed Gaussian processes}
\label{Sec3}

Theorems \ref{theorem_FP-volterra} and \ref{theorem_FP-volterra_general} formulate the FPKE for a time-changed Gaussian process under the assumption that the time-change process is independent of the Gaussian process.  
The case where those processes are dependent is not discussed in this paper.  
As in Proposition \ref{proposition_FP-volterra}, the variance function plays a key role here.

\begin{theorem} \label{theorem_FP-volterra} 
Let $X=(X^1,\ldots,X^n)$ be an $n$-dimensional 
zero-mean Gaussian process with covariance functions $R_{X^j}(s,t)$, $j=1,\ldots,n$,
and let $E^\beta$ be the inverse of a stable subordinator $W^\beta$ of index $\beta\in(0,1)$, independent of $X$.   
 Suppose the variance functions $R_{X^j}(t)=R_{X^j}(t,t)$ are differentiable on $(0,\infty)$ and Laplace transformable.  
Then the transition probabilities $q(t,x)$ of the time-changed Gaussian process $(X_{E^\beta_t})$  satisfy the equivalent PDEs
\begin{align} \label{FP-volterra3-2}
   {D}^\beta_{\ast,t} \hspace{1pt}q(t,x)&=
\sum_{j=1}^n J^{1-\beta}_t \Lambda^\beta_{X^j,t}\hspace{1pt} \partial_{x^j}^2 \hspace{1pt}q(t,x), \ t>0, \ x\in\bbR^n,\\
\label{FP-volterra3-1}
   \partial_t \hspace{1pt}q(t,x)&= 
\sum_{j=1}^n  \Lambda^\beta_{X^j,t}\hspace{1pt} \partial_{x^j}^2 \hspace{1pt}q(t,x), \ t>0, \ x\in\bbR^n,
\end{align}
where $\Lambda^\beta_{X^j,t}$, $j=1,\ldots,n$, are the operators acting on $t$ given by
\begin{align} \label{FP-volterra3.5}
  \Lambda^\beta_{X^j,t}\hspace{1pt} g(t)=
 \frac\beta 2\hspace{1pt} \mL^{-1}_{s\to t} 
   \biggl[ \frac 1{2\pi i} \int_{\bm{\mathcal{C}}}
    (s^\beta-z^\beta)\widetilde{R_{X^j}}(s^\beta-z^\beta) 
   \hspace{1pt}\tilde{g}(z)\hspace{1pt} dz\biggr] (t),
\end{align}
with $z^\beta=e^{\beta \hspace{1pt}\textrm{\rm Ln}(z)}$, $\textrm{\rm Ln}(z)$ being the principal value of the complex logarithmic function $\ln(z)$ 
with cut along the negative real axis,
and 
$\bm{\mathcal{C}}$ being a curve in the complex plane obtained via the transformation $\zeta=z^\beta$ which leaves all the singularities of $\widetilde{R_{X^j}}$ on one side.
\end{theorem}

\begin{proof}
Let $p(t,x)$ denote the transition probabilities of the Gaussian process $X$.   
For each $x\in\bbR^n$, it follows from the independence assumption between $E^\beta$ and $X$ that 
\begin{equation} \label{relation}
   q(t,x)=\int_0^{\infty} f_{E_t^\beta}(\tau)\hspace{1pt}p(\tau,x)\;d\tau, \ t>0.
\end{equation}
Relationship \eqref{relation} and equality \eqref{laplace_Ebeta} together yield 
\begin{align} \label{FP-volterra4}
   \tilde{q}(s,x)=s^{\beta-1}\tilde{p}(s^\beta,x), \ s>0.
\end{align}
Since $R_{X^j}(t)$ is Laplace transformable, 
$\widetilde{R_{X^j}}(s)$ exists for all $s>a$, for some constant $a\ge 0$.   
Taking Laplace transforms on both sides of \eqref{FP-volterra1}, 
\begin{align}  \label{FP-volterra4.5}
   s \,\tilde{p}(s,x)-p(0,x)
   &=\frac 12\sum_{j=1}^n\mL_{t\to s}[  R_{X^j}'(t) \partial_{x^j}^2 \hspace{1pt} p(t,x)](s)\\
   &=\frac 12 \sum_{j=1}^n \bigl(\mL_{t}[R_{X^j}'(t)]\ast \mL_{t}[\partial_{x^j}^2 \hspace{1pt} p(t,x)] \bigr)(s)\notag \\
   &=\frac 12 \sum_{j=1}^n  \biggl[ \frac 1{2\pi i} \int_{c-i\infty}^{c+i\infty} \widetilde{R_{X^j}'}(s-\zeta) \partial_{x^j}^2 \hspace{1pt} \tilde{p}(\zeta,x) \hspace{1pt}d\zeta\biggr]\notag\\
   &=\frac \beta 2 \sum_{j=1}^n\biggl[ \frac 1{2\pi i} \int_{\bm{\mathcal{C}}} \widetilde{R_{X^j}'}(s-z^\beta) \partial_{x^j}^2  \hspace{1pt}  \tilde{p}(z^\beta,x)  z^{\beta-1}\hspace{1pt}dz\biggr], \ s>a, \notag
\end{align}
where $\ast$ denotes the convolution of Laplace images and 
the function 
\begin{align} \label{FP-volterra4.6}
   \widetilde{R_{X^j}'}(s)=s\hspace{1pt}\widetilde{R_{X^j}}(s), \ s>a,
\end{align}
exists by assumption.    
Equation \eqref{FP-volterra4.6} is valid since $R_{X^j}(0)=0$ due to the initial condition $X^j(0)=0$.
Since $E^\beta_0=0$ with probability one, it follows that $p(0,x)=q(0,x)$.  Replacing $s$ by $s^\beta$ and using the identity \eqref{FP-volterra4} yields
\begin{align}\label{FP-volterra5}
   s \,\tilde{q}(s,x)-q(0,x)
   =\frac \beta 2 \sum_{j=1}^n \biggl[ \frac 1{2\pi i} \int_{\bm{\mathcal{C}}}   \widetilde{R_{X^j}'}(s^\beta-z^\beta) \partial_{x^j}^2  \hspace{1pt}\tilde{q}(z,x) \hspace{1pt}dz\biggr], \ s>a^{1/\beta}.
\end{align}
Since the left hand side equals $\mL_{t\to s}[\partial_t \hspace{1pt}q(t,x)](s)$, PDE \eqref{FP-volterra3-1} follows upon
substituting \eqref{FP-volterra4.6}
and taking the inverse Laplace transform on both sides.
Moreover, applying the fractional integral operator $J^{1-\beta}_t$ to both sides of \eqref{FP-volterra3-1} yields \eqref{FP-volterra3-2}.
\end{proof}

The next theorem extends the previous theorem to time-changes which are the inverses of mixtures of independent stable subordinators.

\begin{theorem} \label{theorem_FP-volterra_general} 
Let $X=(X^1,\ldots,X^n)$ be an $n$-dimensional 
zero-mean Gaussian process with covariance functions $R_{X^j}(s,t)$, $j=1,\ldots,n$, and let
$E^\mu$ be the inverse of a mixture $W^\mu$ of independent stable subordinators, independent of $X$.  
Suppose the variance functions $R_{X^j}(t)=R_{X^j}(t,t)$ are differentiable on $(0,\infty)$ and Laplace transformable.  
Then the transition probabilities $q(t,x)$ of the time-changed Gaussian process $(X_{E^\mu_t})$ satisfy the PDEs
\begin{align} \label{FP-volterra13-2}
   D^\mu_{\ast,t} \hspace{1pt}q(t,x)= 
\sum_{j=1}^n 
\int_0^1 J^{1-\beta}_t \Lambda^\mu_{X^j,t}\hspace{1pt} \partial_{x^j}^2 \hspace{1pt} q(t,x) \hspace{1pt} d\mu(\beta), \ t>0, \ x\in\bbR^n,
\end{align}
and
\begin{align} \label{FP-volterra13-1}
   \partial_t \hspace{1pt}q(t,x)=
\sum_{j=1}^n  \Lambda^\mu_{X^j,t}\hspace{1pt} \partial_{x^j}^2  \hspace{1pt}q(t,x), \ t>0, \ x\in\bbR^n,
\end{align}
where $D^\mu_{\ast,t}$ is the operator in \eqref{distributed-derivative} and $\Lambda^\mu_{X^j,t}$, $j=1,\ldots,n$, are the operators acting on $t$ given by
\begin{align}
   \Lambda^\mu_{X^j,t}\hspace{1pt} g(t)
   =\frac {1}{2} \hspace{1pt} \mL^{-1}_{s\to t}  
   \biggl[ \frac 1{2\pi i} \int_{\bm{\mathcal{C}}}
   \bigl(\rho_\mu(s)-\rho_\mu(z)\bigr) \widetilde{R_{X^j}}\bigl(\rho_\mu(s)-\rho_\mu(z)\bigr) 
   \hspace{1pt}m_\mu(z) \hspace{1pt}\tilde{g}(z)\hspace{1pt}dz\biggr] (t), 
\end{align}
with $\rho_\mu(z)=\int_0^1 e^{\beta \hspace{1pt} \emph{Ln} (z)}\hspace{1pt}d\mu(\beta)$, $m_{\mu}(z)=\frac{1}{\rho_\mu(z)}\int_0^1 \beta z^{\beta}\hspace{1pt} d \mu(\beta)$, and 
$\bm{\mathcal{C}}$ being a curve in the complex plane obtained via the transformation $\zeta=\rho_\mu(z)$ which leaves all the singularities of $\widetilde{R_{X^j}}$ on one side.
\end{theorem}

\begin{proof}
We only sketch the proof since it is similar to the proof of Theorem
\ref{theorem_FP-volterra}.  
Let $p(t,x)$ denote the transition probabilities of the Gaussian process $X$.   For each $x\in\bbR^n$, it follows from relationship \eqref{relation} with $f_{E^\beta_t}$ replaced by $f_{E^\mu_t}$ together with equality \eqref{laplace_Emu} that
\begin{align} \label{FP-volterra14}
   \tilde{q}(s,x)=\frac{\rho_\mu(s)}{s}\hspace{1pt}\tilde{p}\bigl(\rho_\mu(s),x\bigr), \ s>0.
\end{align}
Taking Laplace transforms on both sides of \eqref{FP-volterra1} leads to the second to last equality in \eqref{FP-volterra4.5}.  Letting $\zeta=\rho_{\mu}(z)$ yields
\begin{align*} 
   s \,\tilde{p}(s,x)-p(0,x)=\frac 1 2 \sum_{j=1}^n \biggl[ \frac 1{2\pi i} \int_{\bm{\mathcal{C}}} \widetilde{R_{X^j}'}\bigl(s-\rho_\mu(z)\bigr)  \partial_{x^j}^2 \hspace{1pt} \tilde{p}\bigl(\rho_\mu(z),x\bigr)  \frac {\rho_\mu(z)}{z}\hspace{1pt} m_\mu(z) \hspace{1pt}dz\biggr],
\end{align*}
which is valid for all $s$ for which $\widetilde{R_{X^j}}(s)$ exists.
Replacing $s$ by $\rho_\mu(s)$ and using the identity \eqref{FP-volterra14} yields an equation similar to \eqref{FP-volterra5}.  
PDE \eqref{FP-volterra13-1} is obtained upon taking the inverse Laplace transform on both sides.  Finally, applying the fractional integral operator $J^{1-\beta}_t$ and integrating with respect to $\mu$ on both sides of \eqref{FP-volterra13-1} yields \eqref{FP-volterra13-2}.
\end{proof}

\begin{remark}
\begin{em}
a) If $\mu=\delta_{\beta_0}$ with $\beta_0\in(0,1)$, then $\Lambda^\mu_{X^j,t} \hspace{1pt}g(t)=\Lambda^{\beta_0}_{X^j,t}\hspace{1pt}g(t)$ and the FPKEs in \eqref{FP-volterra13-2} and \eqref{FP-volterra13-1}
respectively reduce to the FPKEs in \eqref{FP-volterra3-2} and \eqref{FP-volterra3-1} with $\beta=\beta_0$, as expected. 

b) In Theorem \ref{theorem_FP-volterra_general}, if the components $X^j$ are independent Gaussian processes with a \textit{common} variance function $R_{X}(t)$ which is differentiable and Laplace transformable, then the FPKEs in \eqref{FP-volterra13-2} and \eqref{FP-volterra13-1} respectively reduce to the following simple forms:
\begin{align} 
 D^\mu_{\ast,t} \hspace{1pt}q(t,x) &=  \hspace{1pt}  \int_0^1 J^{1-\beta}_t \Lambda^\mu_{X,t} \hspace{1pt}\varDelta \hspace{1pt} q(t,x)\hspace{1pt} d\mu(\beta);\\
  \partial_t \hspace{1pt}q(t,x)&= \Lambda^\mu_{X,t} \varDelta \hspace{1pt}  q(t,x).
\end{align}

c) As Example \hyperlink{Application-1}{1} in Section \ref{Sec4} shows, Theorem \ref{theorem_FP-volterra} extends the time-fractional order FPKE in \eqref{FBM FPK} for a time-changed fractional Brownian motion to that for a general time-changed Gaussian process, revealing the role of the variance function in describing the dynamics of the process.
\end{em}
\end{remark}

\section{Applications} \label{Sec4}

This section is devoted to applications of the results established in this paper concerning FPKEs for Gaussian and time-changed Gaussian processes.  For simplicity of discussion, we will consider the time-changed process $E^\beta$ given by the inverse of a stable subordinator $W^\beta$ of index $\beta\in(0,1)$, rather than the more general time-change process $E^\mu$.

\vspace{3.2mm}

\noindent
\textsl{Example \hypertarget{Application-1}{1}. Fractional Brownian motion.}
One of the most important Gaussian processes in applied probability is a fractional Brownian motion $B^H$ with Hurst parameter $H\in(0,1)$.  
A one-dimensional \textit{fractional Brownian motion} is a zero-mean Gaussian process with covariance function
\begin{equation}\label{FBM-cov}
R_{B^H}(s,t)=\bbE[B_s^H B_t^H]=\frac{1}{2} (s^{2H}+t^{2H}- |s-t|^{2H}).
\end{equation}
If $H=1/2,$ then 
$B^H$ becomes a usual Brownian motion.  
An $n$-dimensional fractional Brownian motion is
an $n$-dimensional process whose components are independent fractional Brownian motions with a common Hurst parameter.

Stochastic processes driven by a fractional Brownian motion $B^H$ are of increasing interest for both theorists and applied researchers due
to their wide application in fields such as mathematical finance
\cite{Cheridito03}, solar activities \cite{solar} and
turbulence \cite{turbulence}.
The process $B^H,$ like the usual Brownian motion,
has nowhere differentiable paths and stationary increments; however, it does not have independent increments.  
$B^H$ has the integral representation 
$B_t^H=\int_0^t K_H(t,s)\hspace{1pt}d B_s,$ where $B$ is a Brownian motion and $K_H(t,s)$ is a deterministic kernel.  
$B^H$ is not a semimartingale unless $H=1/2$, so the usual It\^o's stochastic calculus is not valid. 
For details of the above properties, see \cite{book,Nualart}.

Let $B^H$ be an $n$-dimensional fractional Brownian motion and let $E^\beta$ be the inverse of a stable subordinator of index $\beta\in(0,1)$, independent of $B^H$.  Then the components of $B^H$ share the common variance function $R_{B^H}(t)=t^{2H}$ and its Laplace transform $\widetilde{R_{B^H}'}(s)=2H\hspace{1pt} \Gamma(2H)/s^{2H}$.  
Hence, Proposition \ref{proposition_FP-volterra} and Theorem \ref{theorem_FP-volterra} immediately recover both the FPKEs in \eqref{FBM FPK} for the fractional Brownian motion $B^H$ and the time-changed fractional Brownian motion $(B^H_{E^\beta_t})$.  
In this case,
\begin{align} \label{correspondence-operators}
J^{1-\beta}_t \Lambda^\beta_{B^H,t}=H G^\beta_{2H-1,t},
\end{align}
where $G^\beta_{\gamma,t}$ is the operator given in \eqref{presentation}.   
Note that the curve $\bm{\mathcal{C}}$ appearing in the expression of the operator $\Lambda^\beta_{B^H,t}$ in \eqref{FP-volterra3.5} can be replaced by a vertical line $\{C+i r \hspace{1pt}; \hspace{1pt}r\in \bbR\}$ with $0<C<s$ since the integrand has a singularity only at $z=s$.

\vspace{3.2mm}

\noindent
\textsl{Example \hypertarget{Application-2}{2}.  Mixed fractional Brownian motion.}  
Let $X$ be an $n$-dimensional process defined by a finite linear combination of independent zero-mean Gaussian processes $X_{1},\ldots, X_{m}$: $X_t=\sum_{\ell=1}^m a_\ell X_{\ell,t}$ with $a_1,\ldots,a_m\in\bbR$. 
For simplicity, assume that for each $\ell=1,\ldots,m$, the components of the vector $X_\ell=(X_\ell^1,\ldots,X_\ell^n)$ share a common variance function $R_{X_\ell}(t)$.  
Then the process $X$ is again a Gaussian process whose components have the same variance function $R_X(t)=\sum_{\ell=1}^m a_\ell^2 R_{X_\ell}(t)$. 
Therefore, it follows from Proposition \ref{proposition_FP-volterra} and Theorem \ref{theorem_FP-volterra} that the FPKEs for $X$ and $(X_{E^\beta_t})$, under the independence assumption between $E^\beta$ and $X$, are respectively given by 
\begin{align}  \label{mixed FPK}
   \partial_t \hspace{1pt}p(t,x)=\phi(t) \hspace{1pt}\varDelta \hspace{1pt}p(t,x) \ \ \ \textrm{and} \ \ \   
   D_{\ast,t}^{\beta}\hspace{1pt}q(t,x)=\Phi^\beta_t \hspace{1pt}\varDelta \hspace{1pt} q(t,x),
\end{align}
where $\phi(t)=\frac 12 \sum_{\ell=1}^m a_\ell^2 R_{X_{\ell}}' (t)$ and $\Phi^\beta_t=\sum_{\ell=1}^m a_\ell^2 J^{1-\beta}_t \Lambda^\beta_{X_{\ell},t}$.  Notice that $\phi(t)$ simply denotes the multiplication by a function of $t$ whereas $\Phi^\beta_t$ is an operator acting on $t$.  
This generalizes the correspondence between the function $t^{2H-1}$ and the operator $G^\beta_{2H-1,t}$ observed in the FPKEs in \eqref{FBM FPK} for the fractional Brownian motion and the time-changed fractional Brownian motion.

A \textit{mixed fractional Brownian motion} is a finite linear combination of independent fractional Brownian motions (see \cite{MRR,Thale} for its properties). 
It was introduced in \cite{Cheridito01} to discuss the price of a European call option on an asset driven by the process.  The process $X$ considered in that paper is of the form $X_t=B_t+a \hspace{1pt}B_t^H$, where $a\in\bbR$, $B$ is a Brownian motion, and $B^H$ is a fractional Brownian motion with Hurst parameter $H\in(0,1)$.   In this situation, the FPKEs in \eqref{mixed FPK}, with the help of \eqref{correspondence-operators}, yield
\begin{align}
   \partial_t \hspace{1pt}p(t,x)&=\frac 12\hspace{1pt}\varDelta \hspace{1pt}p(t,x)+a^2 H t^{2H-1}\hspace{1pt}\varDelta \hspace{1pt}p(t,x); \\ 
   D_{\ast,t}^{\beta}\hspace{1pt}q(t,x)&=\frac 12 \hspace{1pt}\varDelta\hspace{1pt}q(t,x)+a^2 H  G^\beta_{2H-1,t} \hspace{1pt}\varDelta \hspace{1pt} q (t,x). 
\end{align}

\vspace{3.2mm}

\noindent
\textsl{Example \hypertarget{Application-3}{3}. Fractional Brownian motion with variable Hurst parameter.}  
\textit{Volterra processes} form an important subclass of Gaussian processes.  
They are continuous zero-mean Gaussian processes $V=(V_t)$ defined on a given finite interval $[0,T]$ with integral representations of the form $V_t=\int_0^t K(t,s)\hspace{1pt} dB_s$ for some deterministic kernel $K(t,s)$ and Brownian motion $B$ (see \cite{AlosNualart01,Decr-volterra} for details).   Fractional Brownian motions are clearly an example of a Volterra process.  In particular, if $B^H$ is a fractional Brownian motion with Hurst parameter $H\in(1/2,1)$, then $B^H=\int_0^t K_H(t,s)\hspace{1pt} dB_s$ with the kernel
\begin{align} \label{K_H-1}
   K_H(t,s)=c_H s^{1/2-H}\int_s^t (r-s)^{H-3/2}\hspace{2pt}r^{H-1/2} \hspace{1pt}dr, \ \  t>s, 
\end{align}
where the positive constant $c_H$ is chosen so that the integral
$\int_0^{t\wedge s} K_H(t,r)K_H(s,r) \hspace{1pt}dr$ coincides with 
$R_{B^H}(s,t)$ in \eqref{FBM-cov}.
Increments of $B^H$ exhibit long range dependence.

A particular interesting Volterra process
is the fractional Brownian motion with time-dependent Hurst parameter $H(t)$ suggested in Theorem 9 of \cite{Decr-volterra}.
Namely, suppose $H(t): [0,T]\rightarrow (1/2,1)$ is a deterministic function satisfying the following conditions:
\begin{align}\label{condition_star}
    \inf_{t\in[0,T]} H(t)>\frac 12 \ \hspace{3pt} \textrm{and} \ \hspace{3pt}  
    H(t)\in \mathcal{S}_{1/2+\alpha,2} \ \hspace{3pt}\textrm{for some} \ \hspace{3pt}  
     \alpha\in\biggl(0,\hspace{2pt}\inf_{t\in[0,T]}H(t)-\frac 12\biggr), 
\end{align}
where $\mathcal{S}_{\eta,2}$ is the Sobolev-Slobodetzki space given by the closure of the space $C^1[0,T]$ with respect to the semi-norm 
\begin{align}
   \|  f \|^2= \int_0^T\hspace{-1.2mm} \int_0^T \frac{|f(t)-f(s)|^2}{|t-s|^{1+2\eta}}\hspace{1pt} dt \hspace{1pt}ds.  
\end{align}
Then representation \eqref{K_H-1} with $H$ replaced by $H(t)$ induces a covariance function $R_{V}(s,t)=\int_0^{t\wedge s} K_{H(t)}(t,r) \hspace{1pt}K_{H(s)}(s,r)\hspace{1pt}dr$ for some Volterra process $V$ on $[0,T]$.   
The variance function is given by $R_{V}(t)=t^{2H(t)}$ and is necessarily continuous due to the Sobolev embedding theorem, which says $\mathcal{S}_{\eta,2}\subset C[0,T]$ for all $\eta>1/2$ (see e.g.\ \cite{Holmander3}).
Therefore, $H(t)$ is also continuous.

Let $H(t): [0,\infty) \rightarrow (1/2,1)$ be a differentiable function
whose restriction to any finite interval $[0,T]$ satisfies the conditions in  \eqref{condition_star}.   
For each $T>0$, let $K_{V^T}(s,t)$ be the kernel inducing the covariance function $R_{V^T}(s,t)$ of the associated Volterra process $V^T$ defined on $[0,T]$ as above.  
The definition of $K_{V^T}(s,t)$ is consistent; i.e.\ $K_{V^{T_1}}(s,t)=K_{V^{T_2}}(s,t)$ for any $0\le s,\hspace{1pt} t\le T_1\le T_2<\infty$.  
Hence, so is that of $R_{V^T}(s,t)$, which implies that the function $R_X(s,t)$ given by $R_X(s,t)=R_{V^T}(s,t)$ whenever $0\le s,\hspace{1pt} t\le T<\infty$ is 
a well-defined covariance function of a Gaussian process $X$ on $[0,\infty)$ whose restriction to each interval $[0,T]$ coincides with $V^T$.
The process $X$ represents a fractional Brownian motion with variable Hurst parameter.  
The variance function $R_X(t)=t^{2H(t)}$ is differentiable on $(0,\infty)$ by assumption and Laplace transformable due to the estimate $R_X(t)\le t^2$.   Therefore, Proposition \ref{proposition_FP-volterra} and Theorem \ref{theorem_FP-volterra} can be applied to yield the FPKEs for $X$ and the time-changed process $(X_{E^\beta_t})$ under the independence assumption between $E^\beta$ and $X$.

\vspace{3.2mm}

\noindent
\textsl{Example \hypertarget{Application-4}{4}. Fractional Brownian motion with piecewise constant Hurst parameter.}  
The fractional Brownian motion discussed in Example \hyperlink{Application-3}{3} has a continuously varying Hurst parameter $H(t): [0,\infty) \rightarrow (1/2,1)$.  Here we consider a piecewise constant Hurst parameter $H(t): [0,\infty) \rightarrow (0,1)$ which is described as
\begin{align}\label{H(t)}
H(t)=\sum_{k=0}^N H_k \textbf{\textit{I}}_{[T_k,T_{k+1})}(t),
\end{align} 
where $\{H_k\}_{k=0}^N$ are constants in $(0,1)$, $\{T_k\}_{k=0}^N$ are fixed times such that $0=T_0<T_1<\cdots<T_N<T_{N+1}=\infty$, and $\textbf{\textit{I}}_{[T_k,T_{k+1})}$ denotes the indicator function over the interval $[T_k,T_{k+1})$. 

For each $k=0,\ldots, N$, let $B^{H_k}$ be an $n$-dimensional fractional Brownian motion with Hurst parameter $H_k$.  Let $X$ be the process defined by 
\begin{align}
X_t=\sum_{j=0}^{k-1} (B^{H_j}_{T_{j+1}}-B^{H_j}_{T_j})+(B^{H_k}_t-B^{H_k}_{T_k}) \ \ \ \textrm{whenever} \ \ \  t\in[T_k,T_{k+1}).
\end{align}  
Then $X$ is a continuous process representing a fractional Brownian motion which involves finitely many changes of mode of Hurst parameter (described in \eqref{H(t)}).   

The transition probabilities of the process $X$ are constructed as follows.  
For each $k=0,\ldots,N$, let $\theta_k(t)=H_k \hspace{1pt}t^{2H_k -1}$ for $t\in[T_k,T_{k+1})$.  
Let $\{p_k(t,x)\}_{k=0}^N$ be a sequence of the unique solutions to the following initial value problems, each defined on $[T_k,T_{k+1})\times\bbR^n$:
\begin{align}
\partial_t \hspace{1pt} p_0(t,x)&=\theta_0(t) \hspace{1pt} \varDelta \hspace{1pt} p_0(t,x), \ t\in(0,T_1), \ x\in\bbR^n,\\ 
p_0(0,x)&=\delta_0(x), \ x\in\bbR^n,
\end{align}
where $\delta_0(x)$ is the Dirac delta function with mass on $0$, and for $k=1,\ldots, N$,
\begin{align}
\partial_t \hspace{1pt} p_k(t,x)&=\theta_k(t) \hspace{1pt} \varDelta \hspace{1pt} p_k(t,x), \ t\in(T_k,T_{k+1}), \ x\in\bbR^n,\\
 p_k(T_k,x)&=p_{k-1}(T_k-0,x), \ x\in\bbR^n. 
\end{align}
Define functions $\theta(t)$ and $p(t,x)$ respectively by $\theta(t)=\theta_k(t)$ and  $p(t,x)=p_k(t,x)$ whenever $t\in[T_k,T_{k+1})$.  
Then the transition probabilities of $X$ are given by $p(t,x)$ and satisfy
\begin{align}
 \partial_t \hspace{1pt} p(t,x)&=\theta(t) \hspace{1pt} \varDelta \hspace{1pt} p(t,x), \ t\in{\textstyle \bigcup_{k=0}^N}(T_k,T_{k+1}), \ x\in\bbR^n,\\ 
p(0,x)&=\delta_0(x), \ x\in\bbR^n.\\
p(T_k,x)&=p(T_k-0,x), \ x\in\bbR^n, \ k=1,\ldots,N. 
\end{align}
Discussion on existence and uniqueness of the solution to this type of 
initial value problems
is found in \cite{US09}.  
To ensure transition probabilities which are continuous in time, the initial value problem associated with the time-changed process $(X_{E^\beta_t})$ is given by
\begin{align}
 D^\beta_{\ast,t} \hspace{1pt} q(t,x)&=\Theta^\beta_t \hspace{1pt} \varDelta \hspace{1pt} q(t,x), \ t\in(0,\infty), \ x\in\bbR^n,\\ 
q(0,x)&=\delta_0(x), \ x\in\bbR^n,\\
q(T_k,x)&=q(T_k-0,x), \ x\in\bbR^n, \ k=1,\ldots,N,
\end{align}
where $\Theta^\beta_{t}$ is the operator acting on $t$ defined by $\Theta^\beta_{t}=\sum_{k=0}^N H_k \hspace{1pt}G^\beta_{2H_k-1,t}\hspace{1pt} \textbf{\textit{I}}_{[T_k,T_{k+1})}(t)$.

\begin{remark}
\begin{em}
Combining ideas in Examples \hyperlink{Application-3}{3} and \hyperlink{Application-4}{4}, it is possible to construct a fractional Brownian motion having variable Hurst parameter $H(t)\in(1/2,1)$ with finitely many changes of mode and to establish the associated FPKEs.
\end{em}
\end{remark}

\vspace{1mm}

\noindent
\textsl{Example \hypertarget{Application-5}{5}. Ornstein-Uhlenbeck process.}
Consider the one-dimensional Ornstein-Uhlenbeck process $Y$ given by 
\begin{align} \label{example_OU-1}
Y_t=y_0\hspace{1pt}e^{-\alpha t}+\sigma \int_0^t e^{-\alpha(t-s)}\hspace{1pt}dB_s, \ t\ge 0,
\end{align}
where $\alpha\ge 0$, $\sigma>0$, $y_0\in\bbR$ are constants and $B$ is a standard Brownian motion.  
If $\alpha=0$, then $Y_t=y_0+\sigma B_t$, a Brownian motion multiplied by $\sigma$ starting at $y_0$.
Suppose $\alpha>0$.  The process $Y$ defined by \eqref{example_OU-1} is the unique strong solution to the inhomogeneous linear SDE
   \begin{align}
          dY_t=-\alpha Y_tdt +\sigma dB_t \ \ \textrm{with} \ \ Y_0=y_0,  \label{SDE81-1}
     \end{align}
which is associated with the SDE 
   \begin{align}
          d\bar{Y}_t=-\alpha \bar{Y}_tdE^\beta_t +\sigma dB_{E^\beta_t} \ \ \textrm{with} \ \ \bar{Y}_0=y_0,  \label{SDE81-2}
     \end{align}
via the dual relationships $\bar{Y}_t=Y_{E^\beta_t}$ and $Y_t=\bar{Y}_{W^\beta_t}$; 
see \cite{Kobayashi} for details.

Consider the zero-mean process $X$ defined by
\begin{align}
   X_t=Y_t-y_0\hspace{1pt}e^{-\alpha t}=\sigma \int_0^t e^{-\alpha(t-s)}\hspace{1pt}dB_s.
\end{align} 
$X$ is a Gaussian process since each random variable $X_t$ is a linear transformation of the It\^o stochastic integral of the deterministic integrand $e^{\alpha s}$.
Direct calculation yields 
$R_X(t)=\frac{\sigma^2}{2\alpha}(1-e^{-2\alpha t})$ and
$\widetilde{R'_{X}}(s)=\frac{\sigma^2}{s+2\alpha}$.  
Therefore, due to Proposition \ref{proposition_FP-volterra} and Theorem \ref{theorem_FP-volterra}, the initial value problems associated with $X$ and $(X_{E^\beta_t})$, where $E^\beta$ is independent of $X$, are respectively given by 
\begin{align}\label{example_OU-2}
   \partial_t \hspace{1pt}p(t,x)&=\frac{\sigma^2}{2}\hspace{1pt} e^{-2\alpha t}\hspace{1pt}\partial_x^2 \hspace{1pt}p(t,x), \ p(0,x)=\delta_{0}(x);\\
\label{example_OU-3}   D_{\ast,t}^{\beta}\hspace{1pt}q(t,x)&=\frac{\sigma^2 \beta}{2} \hspace{1pt}J^{1-\beta}_t
\mathcal{L}^{-1}_{s \to t}\biggl[\frac{1}{2\pi i} \int_{\bm{\mathcal{C}}}
\frac{\partial_x^2\hspace{1pt}\tilde{q}(z,x)}{s^\beta-z^\beta+2\alpha}\; dz \biggr](t), \ q(0,x)=\delta_{0}(x).
\end{align}
The unique representation of the solution to the initial value problem \eqref{example_OU-2} is obtained via the usual technique using the Fourier transform.  Moreover, expression \eqref{relation} guarantees uniqueness of the solution to \eqref{example_OU-3} as well.

Notice that the two processes $X$ and $(X_{E^\beta_t})$ are unique strong solutions to SDEs \eqref{SDE81-1} and \eqref{SDE81-2} with $y_0=0$, respectively.  Therefore, it is also possible to apply Theorem 4.1 of \cite{HKU} to obtain the following forms of initial value problems which are understood in the sense of generalized functions: 
\begin{align}\label{example_OU-4}
   \partial_t \hspace{1pt}p(t,x)&=\alpha \hspace{1pt}\partial_x \bigl\{x p(t,x) \bigr\} + \frac{\sigma^2}{2}\hspace{1pt} \partial_x^2 \hspace{1pt}p(t,x), \ p(0,x)=\delta_{0}(x);\\
\label{example_OU-5}   D_{\ast,t}^{\beta}\hspace{1pt}q(t,x)&=\alpha \hspace{1pt}\partial_x \bigl\{x q(t,x) \bigr\} + \frac{\sigma^2}{2}\hspace{1pt} \partial_x^2 \hspace{1pt}q(t,x), \ q(0,x)=\delta_{0}(x).
\end{align}
Actually these FPKEs hold in the strong sense as well.  
For uniqueness of solutions to \eqref{example_OU-4} and \eqref{example_OU-5}, see e.g.\ \cite{Friedman} and Corollary 3.2 of \cite{HKU}.

The above discussion yields the following two sets of equivalent initial value problems: \eqref{example_OU-2} and \eqref{example_OU-4}, and \eqref{example_OU-3} and \eqref{example_OU-5}.
At first glance, PDE \eqref{example_OU-2} might seem simpler or computationally more tractable than PDE \eqref{example_OU-4}; however, PDE \eqref{example_OU-3} which is associated with the time-changed process has a more complicated form than PDE \eqref{example_OU-5}.    
A significant difference between PDEs \eqref{example_OU-2} and \eqref{example_OU-4} is the fact that the right-hand side of \eqref{example_OU-4} can be expressed as $A^\ast p(t,x)$ with the \textit{spatial} operator $A^\ast=\alpha \hspace{1pt}\partial_x x + \frac{\sigma^2}{2}\hspace{1pt} \partial_x^2$ whereas the right-hand side of \eqref{example_OU-2} involves both the spatial operator $\partial_x^2$ and the \textit{time-dependent} multiplication operator by $e^{-2\alpha t}$.  This observation suggests: 1) establishing FPKEs for time-changed processes via several different forms of FPKEs for the corresponding untime-changed processes, and 2) choosing appropriate forms for handling specific problems.

\begin{remark}
\begin{em}
Example \hyperlink{Application-5}{5} treated a Gaussian process given by the It\^o integral of the deterministic integrand $e^{-\alpha(t-s)}$.  
Malliavin calculus, which is valid for an arbitrary Gaussian integrator, can be regarded as an extension of It\^o integration (see \cite{Janson,Malliavin,Nualart}). It is known that Malliavin-type stochastic integrals of deterministic integrands are again Gaussian processes.  Therefore, if the variance function of such a stochastic integral satisfies the technical conditions specified in Theorem \ref{theorem_FP-volterra}, then the FPKE for the time-changed stochastic integral is explicitly given by 
\eqref{FP-volterra3-2}, or equivalently, \eqref{FP-volterra3-1}.  
\end{em}
\end{remark}

\bibliographystyle{amsplain}

\end{document}